\documentclass[12pt]{amsart}
\usepackage{amssymb}
\newtheorem{theorem}{Theorem}[section]
\newtheorem{corollary}[theorem]{Corollary}
\newtheorem{lemma}[theorem]{Lemma}
\newtheorem{proposition}[theorem]{Proposition}
\theoremstyle{definition}

\theoremstyle{remark}
\newtheorem{rem}[theorem]{Remark}
\newtheorem{example}[theorem]{Example}

\numberwithin{equation}{section}
\def\ahalf{{\textstyle \frac{1}{2}}}
\def\athird{{\textstyle \frac{1}{3}}}

\def\inpr#1,#2{\t \hbox{\langle #1 , #2 \rangle} \t}
\def\ip<#1,#2>{\langle #1,#2 \rangle}
\def\mod#1{\left \vert #1 \right \vert}

\def\norm#1{\left \Vert #1 \right \Vert}
\def\paren(#1){\left( #1 \right)}

\def\ssnorm#1{\Vert #1 \Vert}
\def\sparen(#1){\Bigl ( #1 \Bigr )}
\def\ssparen(#1){ (#1) }
\def\st{\thinspace : \thinspace}
\def\t{\thinspace}

\def\xchi{{\raise0.15ex\hbox{$\chi$}}}

\newcommand{\U}{\mathcal{U}}
\newcommand{\C}{\mathbb{C}}

\newcommand{\R}{\mathbb{R}}

\newcommand{\<}{\left\langle}
\renewcommand{\>}{\right\rangle}

\newcommand{\AC}{{\rm AC}}
\newcommand{\BV}{{\rm BV}}

\newcommand{\var}{{\rm var}}

\begin{document}

\title{Extensions of an $AC(\sigma)$ functional calculus}%
\author{Ian Doust and Venta Terauds}%
\address{School of Mathematics and Statistics\\
University of New South Wales\\
UNSW Sydney 2052 Australia}%
\email{i.doust@unsw.edu.au}%

\address{School of Mathematical and Physical Sciences\\
University of Newcastle\\
University Drive\\
Callaghan NSW 2308 Australia} \email{Venta.Terauds@newcastle.edu.au}

\thanks{The second author would like to acknowledge the receipt of
an Australian Postgraduate Award. Some of this research was supported
by the Australian Research Council Discovery Grant DP0559097.}%

\date{}%
\begin{abstract}
On a reflexive Banach space $X$, if an operator $T$ admits a
functional calculus for the absolutely continuous functions on its
spectrum $\sigma(T) \subseteq \mathbb{R}$, then this functional
calculus can always be extended to include all the functions of
bounded variation. This need no longer be true on nonreflexive
spaces. In this paper, it is shown that on most classical separable
nonreflexive spaces, one can construct an example where such an
extension is impossible. Sufficient conditions are also given which
ensure that an extension of an $\AC$ functional calculus is possible
for operators acting on families of interpolation spaces such as the
$L^p$ spaces.
\end{abstract}
\maketitle
\section{Introduction}

Given an operator $T$ on a Banach space $X$, it is often important
to be able to identify algebras of functions $\U$ for which one may
sensibly assign a meaning to $f(T)$ for all $f\in\U$. In many
classical situations, the possession of a functional calculus for a
small algebra is enough to ensure an extension of the functional
calculus map to a large algebra. For example, many proofs of the
spectral theorem for normal operators on a Hilbert space first show
that such an operator $T$ must admit a $C(\sigma(T))$ functional
calculus, and then proceed to extend this functional calculus to all
bounded Borel
measurable functions. 

Whether a functional calculus for an operator $T$ has a nontrivial
extension depends crucially on the space on which $T$ acts. For
example, the operator $Tx(t) = t x(t)$ has an obvious $C[0,1]$
functional calculus on both $X=C[0,1]$ and on $X = L^\infty[0,1]$,
but only in the latter case does a nontrivial extension exist. Many
of the positive theorems that exist in this area come as easy
corollaries of theorems which show that a particular functional
calculus is sufficient to ensure that an operator admits an integral
representation with respect to a family of projections. The
integration theory for these families then provides a natural
extension of the original functional calculus. Often however, it is
sufficient to know that one has a large family of projections which
commutes with $T$, and one is less concerned with the topological
properties which are typically required of such families in order to
produce a satisfactory integration theory. In this case, what one is
interested in is whether one can show that $T$ has a functional
calculus for an algebra which contains a large number of idempotent
functions.

An operator with a (norm bounded) functional calculus for the
absolutely continuous functions on a compact set $\sigma \subseteq
\C$ is said to be an $\AC(\sigma)$ operator. (We refer the reader to
\cite{AD1} for the definitions of the function spaces $\AC(\sigma)$
and $\BV(\sigma)$.) In the case where $\sigma \subseteq \R$, such
operators have been more commonly referred to as well-bounded
operators (see \cite{Dow}), although we prefer the more descriptive
term real $\AC(\sigma)$ operators. In this case, as the polynomials
are dense in $\AC(\sigma)$, the $\AC(\sigma)$ functional calculus is
necessarily unique. Real $\AC(\sigma)$ operators were introduced by
Smart and Ringrose \cite{Sm,R} in order to provide a theory which
had similarities to the theory of self-adjoint operators, but which
dealt with the conditionally convergent spectral expansions which
are more common once one leaves the Hilbert space setting.

If $X$ is reflexive, or more generally if the functional calculus
has a certain compactness property, then every real $\AC(\sigma)$
operator $T \in B(X)$ admits an integral representation with respect
to a spectral family of projections on $X$. A real $\AC(\sigma)$
operator with such a representation is said to be of type~(B). The
integration theory for spectral families shows that one may extend
the functional calculus to the idempotent rich algebra $\BV(\sigma)$
of all functions of bounded variation on $\sigma$. That is, there
exists a norm continuous algebra homomorphism $\Psi: \BV(\sigma) \to
B(X)$ such that $\Psi(f) = f(T)$ for all $f \in \AC(\sigma)$.  The
extent to which this theory can be extended to the case where
$\sigma \not\subseteq \mathbb{R}$ is not yet known, and so in this
paper we shall restrict our attention almost exclusively to case
where $\sigma \subseteq \mathbb{R}$. (We refer the reader to
\cite{Dow} for the basic integration theory of real $\AC(\sigma)$
operators.)

There are many examples of $\AC(\sigma)$ operators on nonreflexive
spaces which admit a $\BV(\sigma)$ despite failing to have a
spectral family decomposition. It is natural to ask whether there
are any nonreflexive spaces on which every $\AC(\sigma)$ operator
admits an extended functional calculus, or whether there are any
easily checked conditions which might ensure that such an extension
exists. The aim of this paper is twofold. First we show that if $X$
contains a complemented copy of $c_0$ or a complemented copy of
$\ell^1$, then there is an operator $T \in B(X)$ which admits an
$\AC(\sigma)$ functional calculus which does not have any extension
to $\BV(\sigma)$. In the second half of the paper we shall give
sufficient conditions for an extension of the functional calculus
which apply to linear transformations that act as operators on a
range of $L^p$ spaces.

It was shown in \cite{interpex} that if $(\Omega,\Sigma,\mu)$ is a
finite measure space and $T$ is a real $\AC(\sigma)$ operator on
$L^1(\Omega,\Sigma,\mu)$ and $L^p(\Omega,\Sigma,\mu)$ for any $p >
1$ then $T$ admits a spectral family decomposition on
$L^1(\Omega,\Sigma,\mu)$ and consequently $T$ has a $\BV(\sigma)$
functional calculus on that space. The hypothesis that $T$ be an
$\AC(\sigma)$ operator on some $L^p$ space other than $L^1$ is vital
here; the operator
  \[ Tu(t) = tu(t) + \int_0^t u(s) \, ds, \qquad t \in [0,1], \]
is a real $\AC[0,1]$ operator on $L^1[0,1]$, but it does not admit a
$\BV[0,1]$ functional calculus \cite[p 170]{DdeLaub}. As we shall
show in section~\ref{S-Extrap}, the hypothesis that $\mu(\Omega)$ be
finite can be omitted if one only wishes to deduce the existence of
a $\BV(\sigma)$ functional calculus.

More delicate is the situation for operators acting on $L^\infty$,
and there are several open questions that remain. In practice,
however, concrete operators on this space often have additional
properties which enable one to establish that an extended functional
calculus exists. This will be examined in more detail in
section~\ref{S-Linf}.

Some care needs to be taken in addressing these questions. Even on
Hilbert space, extensions need not be unique. For example, the
operator $T$ on $\ell^2$,
  \[ T(x_0,x_1,x_2,\dots) = (0,x_1,\frac{x_2}{2},\frac{x_3}{3},\dots) \]
admits an $\AC(\sigma(T))$ functional calculus, but both the maps
  \begin{align*}
  \Phi_1(f)(x_0,x_1,x_2,\dots) &= (f(0)x_0,f(1)x_1,f(\tfrac{1}{2}) x_0,\dots) \\
  \Phi_2(f)(x_0,x_1,x_2,\dots)
    &= (\lim_{n \to \infty} f(\tfrac{1}{n})x_0,f(1)x_1,f(\tfrac{1}{2}) x_0,\dots)
  \end{align*}
are bounded algebra homomorphisms from $\BV(\sigma(T))$ to
$B(\ell^2)$ which extend the $\AC(\sigma(T))$ functional calculus.


\section{Nonreflexive spaces on which an extension need not
exist}\label{no-extensions}

There are various examples of real $\AC(\sigma)$ operators in the
literature which do not admit any $\BV(\sigma)$ functional calculus
extension (see, for example, \cite{DdeLaub}). We are not aware
however of any places where the impossibility of an extension is
explicitly proven. Note that this requires more than just showing
that the formula defining $f(T)$ for $f \in \AC(\sigma)$ doesn't
work for $f \in \BV(\sigma)$.

Suppose that $T \in B(X)$ is a real $\AC(\sigma)$ operator which has
a $\BV(\sigma)$ functional calculus. Suppose $\lambda \in \sigma$.
The following standard calculation is based on results such as
Theorem~15.8 of \cite{Dow}, or Theorem 1.4.10 of \cite{LN}.

Let $\xchi_L = \xchi_{\sigma \cap (-\infty,\lambda]}$ and $\xchi_R =
\xchi_{\sigma \cap (\lambda,\infty]}$ so that $\xchi_L + \xchi_R =
1$ in $\BV(\sigma)$. Let $P = P_\lambda = \xchi_L(T)$ and $Q = I-P =
\xchi_R(T)$. Define subsets $L_\lambda, R_\lambda \subseteq
\AC(\sigma)$ by
  \begin{align*}
  L_\lambda &= \{f \in \AC(\sigma) \st \hbox{$f(t) = 0$ for $0 \le t \le
  \lambda$}\}, \\
  R_\lambda &= \{f \in \AC(\sigma) \st \hbox{$f(t) = 0$ for $\lambda \le t \le
  1$}\}.
  \end{align*}

\begin{proposition} The ranges of the projections $P$ and $Q$ satisfy
  \begin{align*}
  PX &\subseteq \{ x \in X \st \hbox{$f(T)x =0$ for all $f \in L_\lambda$}\},
  \\
  QX &\subseteq \{ x \in X \st \hbox{$f(T)x =0$ for all $f \in R_\lambda$}\}.
  \end{align*}
\end{proposition}

\begin{proof}
Note that $f \in L_\lambda \iff f\xchi_L = 0 \iff f = f(1-\xchi_L)$.
Thus
 \[
  Px = x  \iff (I-P)x = 0
           \iff (1-\xchi_L)(T)x = 0 \]
and so if $Px = x$ and $f \in L_\lambda$, then
  $f(T)x = f(T)(1-\xchi_L)(T)x = 0$.
Since $f \in R_\lambda \iff f\xchi_R = 0 \iff f = f(1-\xchi_R)$, the
proof for $Q$ is identical.
\end{proof}

Throughout what follows let $\sigma_0 = \{0\} \cup
\{(-1)^k/k\}_{k=1}^\infty$.

\begin{proposition}\label{c0-example}
 There exists an operator $T$ on $c_0$, which
admits an $\AC(\sigma_0)$ functional calculus but no $\BV(\sigma_0)$
functional calculus.
\end{proposition}

\begin{proof} First note that the map
$U: c_0 \to C(\sigma_0)$,
  \[ U(x_0,x_1,x_2,\dots)(t) =
     \begin{cases}
     x_0,        & t = 0, \\
     x_0+x_k,    & t = \frac{(-1)^k}{k}
     \end{cases}
     \]
is an isomorphism, so it suffices to construct an example on $X =
C(\sigma_0)$.  Define $T \in B(X)$ by $Tx(t) = tx(t)$. Note that
  $\sigma(T) = \sigma_0$.
For $f \in \AC(\sigma(T))$,
  $ f(T)x  = fx$
and so
  \[ \norm{f(T)} \le \norm{f}_\infty \le \norm{f}_{\AC(\sigma_0)}.
  \]
That is, $T$ has an $\AC(\sigma_0)$ functional calculus. Suppose now
that this functional calculus can be extended to a $\BV(\sigma_0)$
functional calculus. Let $\lambda = 0$ and $P = \xchi_{\sigma_0
\cap[-1,0]}(T)$ and $Q = I-P = \xchi_{\sigma_0 \cap (0,\ahalf)}$.
Then (as in the proposition)
  \[ PX \subseteq \{ x \in C(\sigma_0) \st \hbox{$f(T)x =0$ for all $f \in
  L_0$}.\}\]
Now if $f \in L_0$ and $f(T)x = 0$, then $x(-1) = x(-\athird) =
\dots = 0$. As $x \in C(\sigma_0)$, this implies that $x(0) = 0$.
That is, if $Y = \{x \in C(\sigma_0) \st x(0) = 0\}$, then $PX
\subseteq Y$. Similar reasoning shows that $QX \subseteq Y$ too. But
this implies that every element $x  = Px + Qx $ in $C(\sigma_0)$
actually lies in $Y$ which gives the required contradiction.
\end{proof}

One can construct an example on $\ell^1$ in a similar way, although
in this case we need to use the less standard space $\AC(\sigma_0)$
in order to represent the operator in a simple form.

\begin{lemma} $\ell^1$ is isomorphic to $\AC(\sigma_0)$.
\end{lemma}

\begin{proof}
Define $U: \AC(\sigma_0) \to \ell^1$ by
  \[ U(g) = (g(-1),g(\ahalf)-g({\textstyle
  \frac{1}{4}}),g(-\athird)-g(-1),g({\textstyle
  \frac{1}{4}})-g({\textstyle \frac{1}{6}}),\dots) .\]
It is clear that
  \[ \norm{U(g)}_1 \le \mod{g(-1)} + \var_{\sigma_0} g \le
  \norm{g}_{\AC(\sigma_0)}. \]
The inverse map is, writing $x = (x_1,x_2,\dots)$,
  \[ U^{-1}(x)(t) = \begin{cases}
     \sum_{j=1}^n x_{2j-1}, & t = \frac{-1}{2n-1}, \\
     \sum_{j=1}^\infty x_{2j-1}, & t = 0, \\
     \sum_{j=1}^\infty x_j - \sum_{j=1}^{n-1} x_{2j}, & t =
     \frac{1}{2n},\ (n \ne 1), \\
     \sum_{j=1}^\infty x_j,  & t = \ahalf.
     \end{cases}  \]
Note that for this particular set $\sigma_0$, we have that
$\AC(\sigma_0) = \BV(\sigma_0) \cap C(\sigma_0)$. (One can readily
verify this using the results from \cite{AD1}.) Thus in order to
check that a function $g$ is in $\AC(\sigma_0)$, one need only check
that it is of bounded variation, and that $\lim_{t \to 0} g(t)$
exists and equals $g(0)$.   It is easy to check then that the image
of $U^{-1}$ is inside $\AC(\sigma_0)$. Indeed $
\norm{U^{-1}(x)}_{\AC(\sigma_0)} \le 2 \norm{x}_1$, and hence
$U^{-1}$ is continuous.
\end{proof}

The argument given in Proposition~\ref{c0-example} goes through more
or less unchanged if one replaces $C(\sigma_0)$ with
$\AC(\sigma_0)$. Thus there is an example of a real $\AC(\sigma)$
operator on $\ell^1$ whose functional calculus does not extend to
$\BV(\sigma)$.

\begin{lemma} Suppose that $T \in B(X)$ is a real $\AC(\sigma)$ operator
whose functional calculus does not extend to a $\BV(\sigma)$
functional calculus. Let $Y$ contain a complemented copy of $X$.
Then, for a suitable compact set $\sigma'$, there is a real
$\AC(\sigma')$ operator on $Y$ whose functional calculus also fails
to extend to $\BV(\sigma')$.
\end{lemma}

\begin{proof} Let $\sigma' = \sigma \cup \{\omega\}$ where
$\omega = 1+ \max \sigma$. Write $Y = X \oplus Z$ and define $T' \in
B(Y)$ by $T' = T \oplus \omega I_Z$. Then $T'$ clearly has an
$AC(\sigma')$ functional calculus $f(T') = f(T) \oplus f(\omega) I$.
Suppose that this functional calculus admits an extension to
$BV(\sigma')$. The important point to note is that the
characteristic function $\chi_\sigma \in \AC(\sigma')$ and hence the
projection onto $X$, $P = \chi_\sigma(T') = I_X \oplus 0$, commutes
with $f(T')$ for all $f \in \BV(\sigma')$. This implies that we can
write $f(T') = U(f) \oplus V(f)$. Now $\BV(\sigma)$ embeds in a
natural way into $\BV(\sigma')$, and we shall write $\tilde f$ for
the image of $f$ under this embedding. It is easy to check that the
map $\psi:\BV(\sigma) \to B(X)$,  $f \mapsto U(\tilde f)$ is a
continuous Banach algebra homomorphism which extends the original
functional calculus for $T$, contradicting our hypothesis. Hence the
$\AC(\sigma')$ functional calculus for $T'$ can not extend.
\end{proof}

\begin{theorem}\label{main-res} If $X$ contains a complemented copy of $c_0$
or a complemented copy of $\ell^1$, then there exists a real
$\AC(\sigma)$ operator on $X$ for which the functional calculus does
not extend.
\end{theorem}

\begin{rem} This result bears a resemblance to
Theorem~4.4 of \cite{DdeLaub} which shows that under similar
hypotheses on $X$, there exists a real $\AC(\sigma)$ operator on $X$
which is not of type~(B). The operators constructed in that paper
however, do have a $\BV(\sigma)$ calculus. Theorem~4.4 of
\cite{DdeLaub} has been extended to cover an even wider range of
nonreflexive spaces \cite{CD1}, but is it not clear to us how one
might adapt these construction to the present situation.

The hypotheses of the theorem cover most of the classical separable
nonreflexive spaces, but leave the situation for operators on spaces
such as $\ell^\infty$ unclear. At present we do not have any
examples of nonreflexive spaces on which every real $\AC(\sigma)$
operator does have a $\BV(\sigma)$ functional calculus
\end{rem}

\section{Extrapolation to $L^1$}\label{S-Extrap}

Let $(\Omega,\Sigma,\mu)$ be a positive measure space, and write
$L^p$ for $L^p(\Omega,\Sigma,\mu)$.  A linear transformation $T$
defined on equivalence classes of measurable functions $x: \Omega \to \C$ will
be said to define a bounded operator on $L^p$ if $L^p \subseteq
\text{Dom}(T)$ and there exists $K_p<\infty$ such that
$\norm{Tx}_p \le K_p \norm{x}_p$ for all $x \in L^p$.
In this case we shall often write $T_p$ for the restriction of $T$
to $L^p$.


There are two main issues that need addressing when transferring
information about the functional calculus properties of an operator
acting on one space $L^p$ to a second space $L^q$. One concerns the
consistency of the functional calculus. As the above example shows,
the nonuniqueness of extensions means that one cannot expect too
much in general.

Suppose that $p,q\neq \infty$. In the case when the extended
functional calculus comes from a spectral family representation on
each space, then these functional calculi must agree. To see this,
suppose that $\sigma_p$ and $\sigma_q$ are compact subsets of
$\mathbb{R}$. Note that if $T$ defines an $\AC(\sigma_p)$ operator
on $L^p$ and an $\AC(\sigma_q)$ operator on $L^q$, then $T$ has an
$\AC[a,b]$ functional calculus on both spaces for any compact
interval $[a,b]$ containing both $\sigma_p$ and $\sigma_q$ (see
\cite[Section 2]{AD1}). Lemma~3.3 of \cite{interpex} then can be
applied directly to ensure that the spectral families agree on $L^p$
and $L^q$. A consequence of this is that
\begin{enumerate}
  \item $\sigma(T_p) = \sigma(T_q)$ ($=\sigma$ say), and
  \item $T$ defines an $\AC(\sigma)$ operator on both spaces.
\end{enumerate}
Furthermore, if one uses the spectral family to define $\BV(\sigma)$
functional calculi $\Psi_p: \BV(\sigma) \to B(L^p)$ and $\Psi_q:
\BV(\sigma) \to B(L^q)$, then $\Psi_p(f)x = \Psi_q(f)x$ for all $f
\in \BV(\sigma)$ and all $x \in L^p \cap L^q$. (Further details of
the constructions using $\BV(\sigma)$ rather than $\BV[a,b]$ are
available in \cite{As}.)

The following proposition records some standard facts about
operators which act on $L^p$ spaces.

\begin{proposition}\label{extension}
Let $1\leq r < s \leq \infty$ and let $K$ be a positive constant.
Let $S$ be a linear transformation which, for all $p\in (r,s)$,
defines a bounded operator $S_p$ on $L^p$ with $\norm{S_p}_p \leq
K$. Then
\begin{enumerate}
\item[(a)] there is a unique operator $U \in B(L^r)$ with $Ux = Sx$ for
all $x \in \bigcap_{r \le p < s} L^p$, and
\item[(b)] there is an operator $V \in B(L^s)$ with $Vx = Sx$ for
all $x \in \bigcap_{r < p \le s} L^p$.
\end{enumerate}
The operator $U$ in (a) satisfies $\norm{U}_r \le K$. The operator
$V$ in (b) is unique if $s < \infty$ and can be chosen to satisfy
$\norm{V}_s \le K$.
\end{proposition}

\begin{rem} In what follows we shall talk about the `extension' of
$S$ to $L^r$ or $L^s$, but it should be noted that this extension
need not be proper, nor need it (if $s = \infty$) be unique. In
particular, if $L^\infty$ was in the original domain of definition
of $S$, one might have that $V \ne S_\infty$ (see
Example~\ref{eg_diffextns}).
\end{rem}

\begin{proof}
This is a standard exercise except for the case $s = \infty$.

Suppose then that $s = \infty$. We have that $(S_p)^*$ is a bounded
linear operator on $(L^p)^*$ for all $p\in (r,\infty)$, and thus
that there is a linear transformation, $S^*$, defining each bounded
linear operator $S^*_q =: (S_p)^*$ on $L^q$ for $q\in (1,r')$ (where
we use $r'$ for the conjugate exponent of $r$). For all $q\in
(1,r')$, $\norm{S^*}_q \leq K$.

If $y\in \bigcap_{1< q <r'} L^q$ then $ \norm{S^*y}_1 = \lim_{q\to
1^+} \norm{S^*y}_q \leq \lim_{q\to 1^+} K \norm{y}_q  = K \norm{y}_1
\,$ and so $S^*$ may be extended to define a bounded linear operator
$(S^*)_1$ on $L^1$. Let $V = ((S^*)_1)^*$. Clearly
$\ssnorm{V}_\infty \le K$. We want to show that $V$ is an extension
of the linear map $S$.

Suppose then that $x \in \bigcap_{r < p \le \infty} L^p$. Then, for
any $y \in \bigcap_{1 \le q < r'} L^q$,
  \begin{equation}\label{S-hat-eqn}
   \ip<y,Vx> = \ip<S^*y,x> = \ip<y,Sx>.
   \end{equation}
The norm density of $\bigcap_{1 \le q < r'} L^q$ in $L^1$ is now
sufficient to deduce that (\ref{S-hat-eqn}) is true for all $y \in
L^1$, and therefore that $Vx = Sx$.
\end{proof}

The main issue then in wanting to extend the definition of $f(T)$
from one $L^p$ space to another is showing that one does not lose
the property that the map $f \mapsto f(T)$ is an algebra
homomorphism.

\begin{theorem}\label{extfc-left}
Let $1\leq r < s \leq\infty$ and $T$ be a linear transformation
defining real $\AC(\sigma)$ operators, necessarily of type (B),
$T_p$ on $L^p$ for all $p \in (r,s)$. If the $\AC$ functional
calculi for the operators $T_p$ are uniformly bounded (by $M$ say)
for $p \in (r,s)$ then the domain of $T$ can be extended (if
necessary) so that $T$ defines a real $\AC(\sigma)$ operator on
$L^r$. Furthermore, the $AC(\sigma)$ functional calculus for $T_r$
extends to a $BV(\sigma)$ functional calculus.
\end{theorem}

\begin{proof}
The hypotheses imply that for each $p \in (r,s)$, $T_p$ has a
$\BV(\sigma)$ functional calculus $\Psi_p$. As noted at the start of
this section, these maps can be chosen so that
\begin{enumerate}
  \item $\norm{\Psi_p} \le M$ for all $p \in (r,s)$, and
  \item $\Psi_p(f)x = \Psi_q(f)x$ for all $p,q \in (r,s)$, all $f
  \in \BV(\sigma)$ and all $x \in L^p \cap L^q$.
\end{enumerate}
Thus, for each $f \in \BV(\sigma)$, there is a linear transformation
$\Psi(f)$ that defines the operators $\Psi_p(f)$ for all $p\in
(r,s)$. By Proposition \ref{extension}, the domain of each $\Psi_f$
may be extended so that it defines an operator $U_f \in B(L^r)$ with
$\norm{U_f} \le M \norm{f}_{BV}$. As $L^p \cap L^r$ is dense in
$L^r$, it is easy to verify that the map $f \mapsto U_f$ is an
algebra homomorphism from $\BV(\sigma)$ into $B(L^r)$. For example,
if $f,g \in \BV(\sigma)$ and $x \in L^p\cap L^r$,
  \[ U_{fg} x = \Psi_p(fg) x = \Psi_p(f)\Psi_p(g) x = \Psi_p(f)U_g x = U_f U_g x \]
as $U_g x \in L^p \cap L^r$.

Hence $\Psi_r: BV(\sigma) \to B(L^r)$, $\Psi_r(f)=U_f$ defines a
$BV(\sigma)$ functional calculus for $T$ on $L^r$ (and this does
extend the uniquely determined $\AC(\sigma)$ functional calculus for
$T_r$).
\end{proof}

We note that in \cite[Theorem 4.3]{interpex}, it is shown that the
conditions of the above theorem are necessary and sufficient for the
extension of the map $T$ to $L^r$ to be a real $\AC(\sigma)$
operator.

As we shall see in the next section, in the $L^{\infty}$ version of
this result the uniform boundedness is not necessary. Example 5.2 of
\cite{interpex} gives an operator which has an $\BV(\sigma)$
functional calculus on $\ell^p$ for $p\in (1,\infty]$, but for which
the $\AC$ functional calculus is not uniformly bounded on these
spaces.

The Riesz-Thorin interpolation theorem gives the following
application of the theorem. This corollary covers, for example, the
case where $T$ is a real $\AC(\sigma)$ operator on $L^1$ and is
self-adjoint on $L^2$.

\begin{corollary}\label{extraptoL_1}
If, for some $r>1$, a linear map $T$ defines real $\AC(\sigma)$
operators $T_1$ and $T_s$ on $L^{1}$ and $L^{s}$ respectively, then
the operator $T_1$ on $L^{1}$ has a $\BV(\sigma)$ functional
calculus.
\end{corollary}

\begin{proof}
The hypotheses imply that $T$ defines a real $\AC(\sigma)$ operator
on $L^p$ for all $p \in (1,s)$ with a uniform bound on the
$\AC(\sigma)$ functional calculus on these spaces. The result now
follows by Theorem~\ref{extfc-left}.
\end{proof}

\section{Operators on $L^\infty$}\label{S-Linf}

The situation for operators on $L^\infty$ is quite different to that
for operators on $L^1$. The main problem is not that operators can
not be extended to $L^\infty$, but rather that the extensions need
not be unique. In particular, transformations that give rise to
different operators on $L^\infty$ may give identical operators on
another $L^p$ space.

\begin{example}\label{eg_diffextns}
Choose a Banach limit in $L \in (l^\infty)^*$. For $x \in
\ell^{\infty}$, define $Tx := (Lx, 0, 0, \ldots)$. Let $1\leq
r<\infty$; then $T$ defines a bounded linear operator on $\ell^r$
and on $\ell^\infty$. On $\ell^r$, $T$ is the zero operator; on
$\ell^\infty$, $\norm{T}_\infty = 1$. The (non-proper) extension of
$T$ to $\ell^p$ for $r< p <\infty$ is the zero operator and applying
Proposition \ref{extension} to this map on $\ell^p$ for $r\leq p
<\infty$ yields the zero operator on $\ell^\infty$. Of course, the
zero operator is a real $\AC(\{0\})$ operator with a $\BV(\{0\})$
functional calculus. The original map $T$ is not a real
$\AC(\sigma)$ operator on $\ell^\infty$ however, since it is a
nonzero nilpotent operator.
\end{example}

\begin{theorem}\label{extfc-right}
Let $1 < r < \infty$ and $T$ be a linear transformation defining
real $\AC(\sigma)$ operators (of type (B)) $T_p$ on $L^p$ for all $p
\in (r,\infty)$. If the $\AC(\sigma)$ functional calculi for the
operators $T_p$ are uniformly bounded (by $M$ say) for $p \in (r,s)$
then the domain of $T$ can be extended (if necessary) so that $T$
defines a real $\AC(\sigma)$ operator $T_\infty$ on $L^\infty$.
Furthermore, the $\AC(\sigma)$ functional calculus for $T_\infty$
extends to a $\BV(\sigma)$ functional calculus.
\end{theorem}

\begin{proof}
As in Theorem \ref{extfc-left}, for each $f\in\BV(\sigma)$ there
exists a linear transformation $\Psi(f)$ defining operators
$\Psi_p(f)$ for all $p\in (r,\infty)$, and $\Psi_p : \BV(\sigma) \to
B(L^p)$ is a functional calculus for $T_p$.

Using Proposition \ref{extension}, we extend the domain of each map
$\Psi(f)$ so that it defines an operator $V_f\in B(L^\infty)$. By
construction, each $V_f$ is the adjoint of an operator $U_f \in
B(L^1)$, and it may be readily verified that for each $f,g\in
\BV(\sigma)$, $U_{f+g}= U_f + U_g$ and $U_{fg} = U_gU_f$. It follows
that the map $f\mapsto V_f$ is an algebra homomorphism: for example,
for $f,g \in \BV(\sigma)$, $y\in L^1$, $x \in L^\infty$,
\[ \< y, V_{fg}x \> = \< U_{fg}y,x \> = \< U_gU_f y,x \> = \< U_f y, V_g x\> = \<y,V_f V_g x\> \,, \]
so that $V_{fg} = V_f V_g$.

Hence $\Psi_\infty: \BV(\sigma) \to B(L^\infty)$, $\Psi_\infty(f) =
V_f$, defines a $\BV(\sigma)$ functional calculus for $T$ on
$L^\infty$.
\end{proof}

From a practical point of view, the problem is more often to
determine whether a given real $\AC(\sigma)$ operator on $L^\infty$
has a $\BV(\sigma)$ functional calculus. There is in fact no known
example of a real $AC(\sigma)$ operator without a $\BV$ functional
calculus on any $L^\infty$ space. A candidate for such an operator
is
  \[ Tx(t) = tx(t) + \int_t^1 x(s) \, ds, \qquad x \in
  L^\infty[0,1].\]
Showing that this operator does not have a $\BV[0,1]$ functional
calculus would require a different sort of proof to those provided
in Section~\ref{no-extensions} since the functional calculus
\emph{does} extend to the algebra of left-continuous functions of
bounded variation whose continuous singular part is zero. In
particular, and unlike the examples in Section~\ref{no-extensions},
the functional calculus can be extended to include the
characteristic functions $\xchi_{[0,\lambda]}$. It might be noted
that this extension is not constructive. Further details can be
found in \cite[Chapter 15]{Dow}.

Proving an ``$L^\infty$ version'' of Corollary \ref{extraptoL_1} is
problematic. Given a linear transformation $T$ which defines real
$\AC(\sigma)$ operators on $L^\infty$ and $L^r$ for some $r <
\infty$,  extrapolating the operators $\Psi_p(f) = f(T_p)$ ($r \le p
< \infty$) to $L^\infty$ using Theorem~\ref{extfc-right} may not
give a homomorphism which even matches the $\AC(\sigma)$ functional
calculus for $T_\infty$.

\begin{example} As a variant of Example~\ref{eg_diffextns},
consider the linear transformation $Sx := (Lx,Lx,Lx,\dots)$. In this
case $S_p$ is a real $\AC(\{0,1\})$ operator on each $\ell^p$ space.
Indeed each $S_p$ is of type~(B). For $p < \infty$, the
$\BV(\sigma)$ functional calculus for $S_p$ is given by $f(S_p) =
f(0)I$. This has many extensions to $\ell^\infty$, only one of which
is the $\BV(\sigma)$ functional calculus for $T_\infty$,
$f(S_\infty) = f(0)(I-S_\infty) + f(1) S_\infty$. Note that this
example shows that the spectral consistency results listed at the
start of Section~\ref{S-Extrap} do not hold when one of the spaces
is an $L^\infty$ space. In particular, in this example $\sigma(S_p)
\ne \sigma(S_\infty)$.
\end{example}

We finish with two positive results which cover a wide range of
concrete examples.

\begin{proposition}\label{adjoints}
Let $1 \leq r < \infty$ and $T$ be a linear transformation defining
real $AC(\sigma)$ operators $T_r \in B(L^r)$ and $T_\infty \in
B(L^\infty)$. If $T_\infty = S^*$ for some operator $S \in L^1$,
then $T_\infty$ has a $\BV(\sigma)$ functional calculus.
\end{proposition}

\begin{proof}
By the Riesz-Thorin interpolation theorem, $T$ defines real
$\AC(\sigma)$ operators $T_p \in B(L^p)$ for all $p\in [r,\infty)$,
with a uniform bound on the $\AC$ functional calculi of these
operators. Using Proposition \ref{extension} to construct an
operator $U^* \in B(L^\infty)$, we see that $U = S$, as the
extension of the family $(T_p)^*$ to $L^1$ is unique. Thus $U^* =
T_\infty$, and by Theorem \ref{extfc-right}, $T_\infty$ is a real
$\AC(\sigma)$ operator with a $\BV$ functional calculus.
\end{proof}

Proposition~\ref{adjoints} would apply, for example, to the case
where $Tx = Ax$ for some self-adjoint infinite matrix $A$ acting on
$x \in \ell^p$. If $T$ defines a real $AC(\sigma)$ operator on
$\ell^\infty$, and is bounded on $\ell^2$, then, as every
self-adjoint operator on $\ell^2$ is an $AC(\sigma)$ operator, $T$
must admit a $\BV(\sigma)$ functional calculus on $\ell^\infty$.

\begin{proposition}\label{finite-measure}
Suppose that $(\Omega,\Sigma,\mu)$ is a finite measure space. Let $1
\leq r < \infty$ and $T$ be a linear transformation defining real
$\AC(\sigma)$ operators $T_r \in B(L^r)$ and $T_\infty \in
B(L^\infty)$. Then $T_\infty$ has a $\BV(\sigma)$ functional
calculus.
\end{proposition}

\begin{proof}
By interpolation $T$ defines a real $AC(\sigma)$ operator of
type~(B) on $L^p$ for all $p \in [r,\infty)$. Let $\Psi_p$ denote
the $\BV(\sigma)$ functional calculus for $T_p$. Then there exists a
constant $M$ such that $\norm{\Psi_p(f)}_p \le M \norm{f}_{BV}$ for
all $p \in [r,\infty)$. Suppose now that $f \in \BV(\sigma)$. As
$L^\infty \subseteq \bigcap_{r \le p < \infty} L^p$, there exists an
operator $U_f \in B(L^\infty)$ such that $\Psi_p(f)x = U_f x$ for
all $x \in L^\infty$. Further $\norm{U_f}_\infty \le M
\norm{f}_{BV}$. As in the proof of Theorem~\ref{extfc-right}, one
can show that the map $f \mapsto U_f$ is a $BV(\sigma)$ functional
calculus for $T_\infty$.
\end{proof}

%

\end{document}